\documentclass[11pt, a4paper]{article}

\usepackage{amsfonts, latexsym, amssymb, amsthm, amsmath, mathrsfs, esint, color}
\usepackage[latin1]{inputenc}
\usepackage{enumerate}
\usepackage{amsmath}
\usepackage{amsthm}
\usepackage{amssymb}
\usepackage{amscd}
\usepackage{enumerate}
\usepackage{pstricks}
\usepackage{pst-node}
\usepackage{esint}
\usepackage{mathtools}
\usepackage{graphicx}
\usepackage{relsize}
\usepackage{hyperref}
\usepackage{comment}

\def\bs{}
\newcommand{\U}{\mathcal U}

\newcommand{\p}{\varphi}
\newcommand{\la}{\lambda}
\renewcommand{\a}{\alpha}
\renewcommand{\b}{\beta}
\newcommand{\D}{\nabla}

\def\d{\partial}

\def\e{\varepsilon}

\newcommand{\RR}{\mathcal{R}}
\renewcommand{\t}{\widetilde}
\newcommand{\R}{\mathbb{R}}

\def\HH{\mathcal{H}}

\def\AA{\mathcal{A}}

\def\DD{\mathcal{D}}

\newcommand\calY{\mathcal{Y}}
\newcommand{\Y}{\mathcal Y}

\newcommand\wtto{\xrightharpoonup{2}}

\newcommand\hx{x}
\renewcommand\boldsymbol{}

\newcommand\ya{y}
\newcommand\ua{u}

\newcommand\wa{w}
\newcommand\na{n}

\renewcommand{\t}{\widetilde}
\renewcommand{\o}{\tilde}

\renewcommand{\d}{\partial}

\DeclareMathOperator{\Defy}{Def_{\mathcal Y}}

\DeclareMathOperator{\Hessy}{Hess_{\mathcal Y}}

\newcommand\N{\mathbb{N}}
\newcommand\Z{\mathbb{Z}}

\renewcommand\AA{{\mathbf S}}

\newcommand\zetaa{{\boldsymbol \zeta}}
\renewcommand\boldsymbol{}

\renewcommand{\t}{\widetilde}
\renewcommand{\o}{\tilde}

\newcommand{\SO}[1]{\operatorname{SO}(#1)}

\newcommand\sym{\operatorname{sym}}

\newcommand\esssup{\mathop{\operatorname{ess\,sup}}}
\newcommand\ud{\ d}

\newcommand\dist{\operatorname{dist}}
\newcommand\iso{\textrm{iso}}

\DeclareMathOperator{\osc}{osc}

\DeclareMathOperator{\Tr}{Tr}

\def\D{\nabla}
\def\a{\alpha}
\def\b{\beta}

\def\XXint#1#2#3{{\setbox0=\hbox{$#1{#2#3}{\int}$}
     \vcenter{\hbox{$#2#3$}}\kern-.5\wd0}}

\newcounter{bei}

\renewcommand{\S}{\mathbb{S}}
\renewcommand{\o}{\overline}

\newcommand{\zwo}[2]{\begin{pmatrix} {#1}\\{#2} \end{pmatrix}}

\newtheorem{theorem}{Theorem}[section]
\newtheorem{lemma}[theorem]{Lemma}
\newtheorem{proposition}[theorem]{Proposition}

\DeclareMathOperator{\vol}{vol}
\DeclareMathOperator{\graph}{graph}

\begin{document}
\title{Regularity of intrinsically convex $W^{2,2}$ surfaces
and a derivation of a homogenized bending theory 
of convex shells}

\author{
Peter Hornung\footnote{Fachbereich Mathematik, TU Dresden,
01062 Dresden (Germany)
}
\
and
Igor Vel\v ci\'c\footnote{University of Zagreb, Faculty of Electrical Engineering and Computing,
Unska 3, Zagreb (Croatia)
}
}

\date{}

\maketitle
\begin{abstract}
We prove interior regularity for $W^{2,2}$ isometric immersions 
of surfaces endowed with a smooth Riemannian
metric of positive Gauss curvature.
\\
We then derive the $\Gamma$-limit of three dimensional nonlinear
shells with inhomogeneous energy density, in the bending energy regime.
This derivation is incomplete in that it requires an additional technical hypothesis.
\end{abstract}
\vspace{10pt}


\section{Introduction}

For $C^2$ isometric immersions $u$ of a two-dimensional Riemannian manifold with positive Gauss
curvature into $\R^3$, there is a link between the regularity of the metric and the regularity
of $u$; in particular, if the metric is smooth then so is $u$.
Without a priori assumptions on the regularity of $u$ this link is broken.
\\
In the present paper, we show that square integrability of the second fundamental form of $u$ 
is sufficient for the link to persist. In particular, if the metric is smooth,
then $u$ is smooth in the interior, provided that initially it belongs to the Sobolev space $W^{2,2}$.
\\
Our regularity results for metrics with positive Gauss curvature rely upon earlier work by {\v{S}}ver{\'a}k
on the Monge-Amp\`ere equation. Due to the low regularity, 
the passage from the scalar problem to the vectorial problem addressed here is not trivial.
\\
Relaxing $C^2$ regularity to regularity on the Sobolev scale is important for variational
problems: the $W^{2,2}$ isometric immersions studied here arise naturally in thin film elasticity.
In the present paper, we use this regularity result to derive homogenized 
bending models for convex shells from three dimensional nonlinear elasticity. 
\\

Regarding shells theories in elasticity, we refer to \cite{Ciarletshell00}
for an overview of the derivation via formal asymptotic expansions.
In the case of linearly elastic shells, these models can also be justified
rigorously.
\\
More recently, nonlinear models for
rods, curved rods, plates and shells have been derived rigorously
by means of $\Gamma$-convergence,
starting from three dimensional nonlinear
elasticity. The first results in that direction can be found in
\cite{AcBuPe,LDRa95, LDRa96}.
The nonlinear bending theory for plates
was derived in \cite{FJM-02}, and the corresponding theory
for shells in \cite{FJMMshells03}.
\\
In the second part of this article we derive a homogenized nonlinear bending
theory of shells, by simultaneous homogenization and dimension reduction.
This generalizes the results from \cite{FJMMshells03}.
Our starting point is the energy functional of three dimensional nonlinear elasticity:
We consider a reference configuration which is a
shell $S^h \subset \R^3$ of thickness $h > 0$ around an embedded
surface $S\subset\R^3$. The elastic energy stored in the deformed configuration
determined by a deformation $u\in W^{1,2}(S^h, \R^3)$
is given by
\begin{equation}\label{intro:1}
\frac{1}{h^2\,|S^h|}\int_{S^h}W_\e(x,\nabla \bs u(x))\,dx.
\end{equation}
The function $W_\e$ is a stored energy function that oscillates periodically
in $x$, with some period $\e\ll 1$. We are interested in the
effective behaviour of the functionals \eqref{intro:1}
when both the thickness $h$ and the period $\e$ are small:
we consider the asymptotic behaviour of \eqref{intro:1}
when $h$ and $\e$ tend to zero simultaneously.
\\
Such a combination of dimension reduction and homogenization
was studied, e.g., in \cite{Braides-Fonseca-Francfort-00}.
More recently, homogenized nonlinear plate theories in the von K\'arman energy regime and in the bending
regime were studied in \cite{NeuVel-12} and in \cite{Horneuvel12,Vel13}.
In these cases one does not obtain
an infinite-cell homogenization formula as in the membrane case studied in
\cite{Braides-Fonseca-Francfort-00}. This is because
for small strains the energy is essentially convex, so one can use two-scale convergence techniques.
\\

The derivation of a homogenized theory of shells in the von K\'arm\'an energy regime
was carried out in \cite{Hornungvel12}. Different models
were obtained in the regime $h \ll\e$.
For generic shells, the models for the situations $\e^2 \lesssim h\ll \e$ 
have been derived.
For convex shells, the whole regime $h \ll \e$ is now understood.
\\
The geometric framework developed in \cite{Hornungvel12} will be used in the present paper as well.
Here we are interested in the analogous theory for the bending energy regime. We
restrict ourselves to convex shells. Our main
result in this direction is Theorem \ref{tm:glavni}. 
The derivation of the lower bound is quite natural. However,
as usual, we can prove sharpness of the lower bound only
for regular limiting deformations.
We are not able to close this regularity gap. However, our regularity result
Theorem \ref{thm1} allows us to narrow the gap: using it, we can construct the
required recovery sequence starting from a limiting deformation which is not
in $W^{3, \infty}$, but merely in $W^{2, \infty}$. In addition, 
Theorem \ref{thm1} confirms
the intuition that all finite energy deformations of a convex shell preserve
convexity.

\section{Regularity of intrinsically convex $W^{2,2}$ surfaces}

The purpose of this chapter is to prove the following result:

\begin{theorem}\label{thm1}
	Let $U\subset\R^2$ be open and 
	let $g\in C^{\infty}( U, \R^{2\times 2}_{\sym})$ be a smooth Riemannian metric on $U$. 
	Assume that the Gauss curvature $K$ of $g$ is everywhere positive 
	and let $u : U\to\R^3$ belong to the space
	$$
	W^{2,2}_g(U) = 
	\left\{
	u\in W^{2,2}(U, \R^3) : (\D u)^T(\D u) = g \mbox{ almost everywhere on }U
	\right\}.
	$$
	Then $u\in C^{\infty}(U)$.
\end{theorem}

In the statement of this theorem and elsewhere, we always refer to the precise
representative of the Sobolev functions in question.
\\
To prove Theorem \ref{thm1} we use ideas and a key result from the unpublished (but widely circulated) manuscript \cite{sverak}. 
For our purposes, its main result is to deduce convexity of $W^{2,2}$
solutions $f$ of the Monge-Amp\`ere inequality $\det \D^2 f \geq c > 0$,
cf. Lemma \ref{le-sverak} below.
In \cite{sverak}, this result is combined with a local graphical representation 
to prove smoothness of $C^{1,1}$ isometric immersions of subdomains of the sphere, 
endowed with the standard metric.
\\
Our proof of Theorem \ref{thm1} also uses this idea of representing $u$ locally as a graph of a function $f$. However, a priori $u$ is not $C^1$.
Instead, we show that the 
normal $n_u$ to $u$ is continuous. It is defined by
$$
n_u = \frac{\d_1 u\times\d_2 u}{|\d_1 u\times\d_2 u|}.
$$
It turns out that continuity of the normal
is enough to ensure that $u$ be locally a $C^1$ graph.
\\
Finally, a bootstrap argument, using classical facts about Monge-Amp\`ere equations
on one hand and exploiting the link between $u$ and its graphical representation on the other hand, 
implies that $u$ is smooth.

\subsection{Continuity of the normal}

The purpose of this section is to provide a fairly self-contained proof of Proposition \ref{pro1} below.
In doing so, we combine ideas from \cite{BrezisNirenberg1, BrezisNirenberg2} 
and others, and we introduce a suitable notion of topological degree. For the reader's convenience,
we include proofs of its relevant properties.

In what follows, we use the notation $|g| = \det g$. The Christoffel symbols of $g$ are denoted by $\Gamma_{\a\b}^{\gamma}$. The Gauss curvature of the Riemannian metric
$g$ is denoted by $K$. By $B_R$ we denote the open ball of radius $R$ in $\R^2$
centered at the origin. And $U\subset\R^2$ is an open set unless specified
otherwise.

Define $h : U\to\R^{2\times 2}_{\sym}$ by $h = n_u\cdot\D^2 u$.
The Gauss equation is easily seen to remain true for $u\in W^{2,2}$. It reads:
\begin{equation}
\label{gk-5}
\d_{\a}\d_{\b} u = h_{\a\b} n_u + \Gamma_{\a\b}^{\gamma}\d_{\gamma} u.
\end{equation}
Since we are dealing with $W^{2,2}$ maps, we should verify the validity
of Gauss' Theorema Egregium.

\begin{lemma}\label{egregium}
If $u\in W^{2,2}_g(U)$ and $h = n_u\cdot\D^2 u$, then $\det h = K |g|$ almost everywhere on $U$.
\end{lemma}
\begin{proof}
As shown in \cite[Proof of Proposition 6]{FJM-06}, by approximation it is easy to see
that the map $u$ satisfies
\begin{equation}
\label{egr-1}
|\d_1\d_2 u|^2 - \d_1\d_1 u\cdot \d_2\d_2 u = 
\frac{1}{2}\left( \d_2\d_2 g_{11} + \d_1\d_1 g_{22} - 2 \d_1\d_2 g_{12} \right)
\end{equation}
almost everywhere on $U$. Denote by $P(x)$ the orthogonal projection
from $\R^3$ onto the subspace spanned by $\d_1 u(x)$ and $\d_2 u(x)$.
Then we deduce from \eqref{egr-1} that
\begin{equation}\label{egr-2}
\det h = - |P(\d_{1}\d_2 u)|^2 + P(\d_1\d_1 u)\cdot P(\d_2\d_2 u)
+ \frac{1}{2}\left( \d_2\d_2 g_{11} + \d_1\d_1 g_{22} - 2 \d_1\d_2 g_{12} \right).
\end{equation}
But in view of \eqref{gk-5} we have
$$
P(\d_{\a}\d_{\b} u)\cdot P(\d_{\gamma}\d_{\delta} u) = \Gamma_{\a\b}^{\rho}\Gamma^{\sigma}_{\gamma\delta}
g_{\rho\sigma}.
$$
We conclude that the right-hand side of \eqref{egr-2} can be computed from $g$ and its derivatives.
Since $g$ is smooth, a classical computation therefore 
shows that the right-hand side of \eqref{egr-2} agrees with $K  |g| $. 
\end{proof}

\begin{proposition}\label{pro1}
Let $g\in C^{\infty}(U)$ be a smooth Riemannian metric on $U$. 
Assume that the Gauss curvature $K$ of $g$ is positive 
on $U$ and let $u\in W^{2,2}_g(U)$.
Then the normal $n_u$ to $u$ is continuous on $U$.
\end{proposition}

In order to prove Proposition \ref{pro1}, we will introduce and prove some results
about the topological degree of $\S^2$-valued maps.
\\
So let $\p\in W^{1,2}(B_{R}, \S^2)$.
Then, for almost every $r\in (0, R)$ we have
$
\p|_{\d B_r}\in W^{1,2}(\d B_r)$, hence by Sobolev embedding
\begin{equation}
\label{p1-1}
\p|_{\d B_r}\in C^0(\d B_r).
\end{equation}
By a classical result of Schoen and Uhlenbeck,
there exist $\p_k\in C^{\infty}(B_R, \S^2)$ 
converging strongly in $W^{1,2}$ to $\p$ as $k\to\infty$.
After possibly passing to a subsequence, we may assume that
$\p_k\to\p$ in $W^{1,2}(\d B_r)$ for almost every $r\in (0, R)$. Hence
for such $r$
\begin{equation}
\label{p1-2}
\p_k\to \p \mbox{ uniformly on }\d B_r.
\end{equation}
In fact, setting $f_k(r) = \int_{\d B_r}|\D\p_k - \D\p|^2$, by the coarea formula we have
$$
\|f_k\|_{L^1(B_R)} = \int_0^R\ dr\ \int_{\d B_r}|\D\p_k - \D\p|^2 = \int_{B_R} |\D\p_k - \D\p|^2\to 0.
$$
Hence there is a subsequence such that $f_{k_j}(r)\to 0$ for almost every
$r\in (0, R)$.
\\
For $\p\in W^{1,2}(B_R, \S^2)$ we denote by $\RR_{\p}$ the set of
those $r\in (0, R)$ such that \eqref{p1-1} is satisfied
and such that, in addition, there
exist $\p_k\in C^{\infty}(B_R, \S^2)$ converging strongly in $W^{1,2}$
to $\p$ and satisfying \eqref{p1-2}.
Note that $\p(\d B_r)$ is compact for such $r$, due to \eqref{p1-1}.
\\
For $r\in\RR_{\p}$ define the degree
$Q : \S^2\to\R$ of $\p$ with respect to $B_r$ by setting
\begin{equation}
\label{p1-3}
Q(y) = \int_{B_r}\p^*\eta,
\end{equation}
where $\eta$ is any smooth $2$-form on $\S^2$ with $\int_{\S^2}\eta = 1$
which is supported in the connected
component $\Lambda_y$ of $\S^2\setminus \p(\d B_r)$ that contains $y$.
\\
We claim that $Q$ is well-defined, i.e., that it is independent of the choice of $\eta$.
We use the following well-known fact.
\begin{lemma}
Let $\Lambda\subset\S^2$ be connected and let 
$\t\eta$ be a smooth $2$-form on $\S^2$ whose support is contained in $\Lambda$
and which is such that $\int_{\S^2}\t\eta = 0$. Then there exists 
a smooth $1$-form $w$ on $\S^2$ with support in $\Lambda$ and such that 
$\t \eta = dw$.
\end{lemma}

In view of the lemma it remains to show that if $w$ is a smooth $1$-form 
supported in $\Lambda_y$ then $\int_{B_r}\p^*(dw) = 0$.
\\
Since $dw$ is a $2$-form and since $\p_k\to\p$ strongly in $W^{1,2}(B_r)$, we see that
$$
\p_k^*(dw)\to \p^*(dw) \mbox{ strongly in }L^1(B_r).
$$
Hence 
\begin{equation}
\label{p1-4}
\int_{B_r} d(\p_k^*w) = \int_{B_r} \p_k^*(dw) \to \int_{B_r} \p^*(dw).
\end{equation}
Due to \eqref{p1-2}, the compact set $\p_k(\d B_r)$ does not intersect the support of
$w$ for $k$ large enough, because the latter has positive distance from the compact
set $\p(\d B_r)$. Therefore, $\p_k^* w$ has compact support in $B_r$. Hence,
by Stokes' theorem, the left-hand side of \eqref{p1-4} is zero. This concludes the proof
showing that $Q$ is well-defined by \eqref{p1-3}.
\\
Recall that the essential range of $\p|_{B_r}$ is the smallest closed set
$F$ such that $\p(x)\in F$ for almost every $x\in B_r$; as shown in \cite{BrezisNirenberg1}
it is well-defined.
More or less directly from the definition of $Q$, we see the following:

\begin{lemma}\label{le2}
Let $\p\in W^{1,2}(B_R, \S^2)$, let $r\in\RR_{\p}$
and define $Q$ as in \eqref{p1-3}.
Then the following are true:
\begin{enumerate}[(i)]
\item \label{p1-6}
$Q$ is constant on every connected component of $\S^2\setminus\p(\d B_r)$;
\item \label{p1-7}
If $Q(y)\neq 0$ then $\Lambda_y$ is contained in the essential range $F$ of
$\p|_{B_r}$.
\item \label{p1-7b}
$Q$ takes integer values.
\end{enumerate}
\end{lemma}
\begin{proof}
To prove \eqref{p1-7}, assume that $\Lambda_y$ is not contained in $F$. Then
there exists an open set $\Lambda\subset\Lambda_y\setminus F$
and a normalized smooth $2$-form $\eta$ supported on $\Lambda$.
So $\p^*\eta = 0$ almost everywhere on $B_r$. Hence we would have $Q(y) = 0$.
\\
To prove \eqref{p1-7b} just note that the last convergence in \eqref{p1-4}
is also true for any other $2$-form; in particular for the form $\eta$ in \eqref{p1-3}.
But for smooth $\p$, the right-hand side of \eqref{p1-3}
is known to attain only integer values.
\end{proof}

\begin{lemma}\label{le3}
Let $u\in W^{2,2}_g(B_R)$, let $r\in\RR_{n_u}$
and define $Q$ as in \eqref{p1-3} with $\p = n_u$.
Then $Q \geq 0$ on $\S^2\setminus n_u(\d B_r)$.
Moreover, if $y\in \S^2\setminus n_u(\d B_r)$ is such that
$Q(y) = 0$, then $\Lambda_y$ does not intersect the essential
range of $n_u|_{B_r}$.
\end{lemma}
\begin{proof}
Denote by $\eta_{\S^2}$ the standard area form on $\S^2$. 
Then $n_u^*\eta_{\S^2} = K d\vol_g$, due to Lemma \ref{egregium}.
Applying \eqref{p1-3} with $\eta = \rho\eta_{\S^2}$, 
we see that
$$
\left(\int_{B_r} \rho\eta_{\S^2}\right)\ Q(y) = \int_{B_r} (\rho\circ n_u) K\ d\vol_g. 
$$
for every $\rho\in C^{\infty}_0(\Lambda_y)$.
Since $K > 0$ on $B_r$, we conclude that if $Q(y) = 0$ then
$\rho\circ n_u = 0$ almost everywhere on $B_r$. Since $\rho$ was arbitrary,
this implies that $n_u (x)\in \S^2\setminus \Lambda_y$ for almost every $x\in B_r$.
Since $\S^2\setminus\Lambda_y$ is closed, by minimality of the essential range
we conclude that it must be contained in $\S^2\setminus\Lambda_y$.
\end{proof}

\begin{lemma}\label{le4}
Let $u$ and $Q$ be as in the hypotheses of Lemma \ref{le3}.
If $r\in\RR_{n_u}$ is small enough, then $Q$ is zero at some point in 
$\S^2\setminus n_u(\d B_r)$.
\end{lemma}
\begin{proof}
There exists a constant $C$ depending only on $g$ such that
$$
\int_{B_r} K\ d\vol_g \leq Cr^2.
$$
We choose $r > 0$
so small that the right-hand side is bounded by $1/4$ times the area of $\S^2$.
\\
Assume for contradiction that $Q\neq 0$ 
everywhere on $\S^2\setminus n_u(\d B_r)$. Then by Lemma \ref{le3}
we know that $Q$ is positive and so by Lemma \ref{le2} we have $Q\geq 1$ on $\S^2\setminus n_u(\d B_r)$.
\\
Since $n_u\in W^{1,2}(\d B_r)$, it maps $\d B_r$ into a set of zero area, cf. 
\cite{Reshetnyak}.
So the area of $\S^2\setminus n_u(\d B_r)$ is that of $\S^2$.
Hence there exist finitely many pairwise disjoint connected components $\Lambda_1, ..., \Lambda_M$ of
$\S^2\setminus n_u(\d B_r)$ and $\psi_i\in C^{\infty}_0(\Lambda_i)$ 
taking values in $[0, 1]$ and such that
\begin{equation}
\label{le4-1}
\sum_{i = 1}^M \HH^2\left( \{\psi_i = 1\} \right) \geq \frac{1}{2}\HH^2(\S^2).
\end{equation}
Here $\HH^2$ denotes the $2$-dimensional Hausdorff measure in $\R^3$.
Let $y_i\in \Lambda_i$ and note that $\sum_i\psi_i \leq 1$ pointwise on $\S^2$. Hence, recalling that
$Q(y_i) \geq 1$,
\begin{align*}
\int_{B_r} K\ d\vol_g &\geq \sum_{i = 1}^M \int_{B_r} (\psi_i\circ n_u)\ K\ d\vol_g
\\
&= \sum_{i = 1}^M \int_{B_r} n_u^*(\psi_i\eta_{\S^2})
\\
&= \sum_{i = 1}^M Q(y_i) \int_{\S^2} \psi_i\eta_{\S^2}
\\
&\geq \sum_{i = 1}^M \int_{\S^2} \psi_i\eta_{\S^2}
\\
&\geq \sum_{i = 1}^M \HH^2\left( \{\psi_i = 1\} \right).
\end{align*}
In view of \eqref{le4-1} this contradicts our choice of $r$.
\end{proof}

\begin{proof}[Proof of Proposition \ref{pro1}]
Since $Q$ vanishes
at some point by Lemma \ref{le4}, by Lemma \ref{le3}
it is in fact zero on a whole connected component $\Lambda$ of
the relatively open set $\S^2\setminus n_u(\d B_r)$
and (after possibly redefining $n_u$ on a set of measure
zero) $n_u$ does not take values in $\Lambda$.
\\
We assume without loss of generality that $e_3\in \Lambda$ and we denote by
$\Psi : \S^2\setminus\{e_3\}\to\R^2$
the stereographic projection. Since $\Lambda\subset \S^2\setminus n_u(\d B_r)$
is relatively open, 
there exists $\rho > 0$ such that
$\S^2\cap B_{2\rho}(e_3)$ does not intersect $n_u(\o B_r)$. And 
$\Psi\in C^{\infty}\left(\R^3\setminus \o{B_{\rho}(e_3)}\right)$.
Hence $\Psi\circ n_u\in (W^{1,2}\cap L^{\infty})(U, \R^2)$. Since $\Psi$
is conformal, we deduce from $K > 0$ 
that the Jacobian of $\Psi\circ n_u$ 
does not change its sign and is bounded
away from zero. Hence $\Psi\circ n_u$ is continuous, cf. \cite{Reshetnyak}. 
Hence $n_u$ is continuous as well.
\end{proof}

\subsection{Immersions with continuous normal}

Deviating from our general notation, in the next proposition
$g$ will denote an arbitrary continuous Riemannian metric.

\begin{proposition}\label{baeh}
Let $g\in C^0(\o U, \R_{\sym}^{2\times 2})$ be a continuous
Riemannian metric and let $u\in W^{2,2}_g(U)$. Then $u$ is locally Bilipschitz.
\\
More precisely, there exists $R_0 > 0$ such that
for every $x\in U$
we have
$$
\frac{\la_1(x)}{2} |z-y|\leq |u(z) - u(y)| 
\leq \|\Tr g\|^{1/2}_{L^{\infty}(U)}|z-y|
\mbox{ for all }z, y\in B_{R}(x).
$$
Here $\la_1(x)$ is the smallest eigenvalue of $g(x)$ and $R=\min\{R_0,\frac{1}{8} \dist_{\d U}(x)\}$.
\end{proposition}
\begin{proof}
We follow \cite{H-invertibility}, which in turn follows \cite{MullerSverak}.
Clearly $u$ is Lipschitz, because $|\D u|^2 = \Tr g$ is uniformly
bounded.
\\
Now let $\e\in (0, 1)$ and choose $R_0 > 0$ such that
\begin{equation}
\label{baeh-0}
\osc_{B_{R}(x)} g +
\left(\int_{B_{R}(x)} |\D^2 u|^2\right)^{1/2} < \e
\end{equation}
whenever $x\in U$ and $R\leq 8R_0$ and $B_R(x)\subset U$. 
\\
Fix one such pair $x$ and $R$ and consider two distinct points in $B_{R/8}(x)$.
They are a distance $L\in (0, {R}/{4})$ apart. After rotation and translation,
we may assume that they agree with the origin and
the point $(L, 0)$, respectively.
We may also assume that $\D u(t, 0)$ exists and that
$(\D u)^T(\D u)(t, 0) = g(t, 0)$ for $\mathcal{L}^1$ almost every
$t\in [0, L]$, and that
$$
u(L, 0) - u(0, 0) = \int_0^L \d_1 u(t, 0)\ dt.
$$
(In fact, for almost every $a\in (-R/50, R/50)$
the analogous statements are true with $u(\cdot, a)$ instead of $u(\cdot, 0)$.
So we can apply the following proof to each of these maps and then let $a\to 0$.)
\\
For brevity we write $G = \D u$ and
$
\o G = \frac{1}{L}\int_0^L G(t, 0)\ dt.
$
By the Trace Theorem and Poincar\'e's inequality
there exists a constant $C_0$ such that
\begin{equation}
\label{baeh-1}
\frac{1}{L}\int_0^L |G(t, 0) - \o G|^2\ dt
\leq C_0\int_{(0, L)^2} |\D G|^2\leq C_0\e^2,
\end{equation}
where we have used $\D G = \D^2 u$ and \eqref{baeh-0}.
On the other hand, since $G^T G = g$, we have
$$
|\o G^T\o G - g| \leq (|\o G| + |G|)|G - \o G|
\leq 2\|\Tr g\|_{L^{\infty}}^{1/2}|G - \o G|.
$$
Hence using \eqref{baeh-1} and Jensen's inequality,
\begin{align*}
\frac{1}{L}\int_0^L |\o G^T\o G - g(t, 0)|\ dt
&\leq
2\|\Tr g\|_{L^{\infty}}^{1/2}
\cdot\frac{1}{L}\int_0^L |G(t, 0) - \o G|\ dt
\leq C_2\e,
\end{align*}
where
$
C_2 = 2\|\Tr g\|_{L^{\infty}}^{1/2} C_0^{1/2}.
$
Using \eqref{baeh-0} we conclude that
$$
|\o G^T\o G - g(0, 0)|\leq
\osc_{B_R(x)} g + C_2\e
\leq (C_2 + 1)\e.
$$
Hence choosing 
$$
\e = \frac{g_{11}(0, 0)}{2(C_2 + 1)},
$$
we have
$
|\o G e_1|^2 \geq \frac{|g_{11}(0, 0)|^2}{4}.
$
Thus
\begin{align*}
|u(L, 0) - u(0, 0)| &= 
\Big| \int_0^L \d_1 u(t, 0)\ dt \Big|
= L|\o Ge_1|\geq \frac{|g_{11}(0, 0)|}{2}\cdot L.
\end{align*}
\end{proof}

The hypotheses of the following lemma are satisfied by isometric
immersions with a continuous normal.

\begin{lemma}\label{muuh}
Let $u\in W^{1,\infty}(U, \R^3)$ be an immersion
and assume that its
normal $n_u$ is continuous on $\o U$. If $x\in U$ and
$R\leq\dist_{\d U}(x)$, then
$$
\Big|
n_u(x)\cdot \left( u(z) - u(y) \right)
 \Big| 
\leq \|\D u\|_{L^{\infty}(U)}\cdot\big(\osc_{B_{R}(x)} n_u\big)\cdot|y-z|
$$
for all $z$, $y\in B_{R}(x)$.
\end{lemma}
\begin{proof}
We may assume that $y$ agrees with the origin and $z = (L, 0)$ 
for some $L\in (0, 2R)$. 
As in the proof of Proposition \ref{baeh},
we may also assume that $\D u(t, 0)$ exists for $\mathcal{L}^1$
almost every $t\in [0, L]$, and that
$$
u(L, 0) - u(0, 0) = \int_0^L \d_1 u(t, 0)\ dt.
$$
Hence the claim follows at once from the equation
\begin{align*}
n_u(x)\cdot \left( u(L, 0) - u(0,0) \right) &=
\int_0^L \d_1 u(t, 0)\cdot \left( n_u(x) - n_u(t, 0)\right)\ dt.
\end{align*}
\end{proof}

\subsection{Proof of Theorem \ref{thm1}}

Assume that the hypotheses of Theorem \ref{thm1} are satisfied
and fix a point $x_0\in U$. 
We will prove that $u$ is smooth in a neighbourhood of $x_0$.
By Proposition \ref{pro1} the normal $n_u$ is continuous.
We assume without loss of generality that $n_u(x_0) = e_3$ and we write
$u = \zwo{\Psi}{V}$, where $V = e_3\cdot u$ and $\Psi : U\to\R^2$ is the in-plane component.

\begin{lemma}\label{psibi}
There exists $r > 0$ such that $\Psi$ is
Bilipschitz on $B_r(x_0)$.
\end{lemma}
\begin{proof}
Clearly $\Psi$ is Lipschitz because
so is $u$. Let $R_0$ be as in Proposition \ref{baeh}
and denote by $\la_1$ the smallest eigenvalue of $g(x_0)$.
Let $R\leq R_0$ be such that $B_{R}(x_0)\subset U$ and
$$
\osc_{B_{R}(x_0)} n
\leq\frac{\la_1}{4\|\Tr g\|_{L^{\infty}(U)}^{1/2}}.
$$
Set $r = R/8$. Then by Proposition \ref{baeh} and Lemma \ref{muuh},
for all $y$, $z\in B_r(x_0)$ we have 
\begin{align*}
|\Psi(z) - \Psi(y)|^2 &= |u(z) - u(y)|^2 - |V(z) - V(y)|^2
\\
&\geq \left(\frac{\la_1^2}{4} - \frac{\la_1^2}{16}\right)|z-y|^2
= \frac{3\la_1^2}{16}|z - y|^2.
\end{align*}
\end{proof}

We may assume without loss
of generality that $\Psi(x_0) = 0$.
By Lemma \ref{psibi}, after possibly shrinking $U$ we may assume that
$\Psi$ is a Bilipschitz homeomorphism from $U$ onto $\Psi(U)$
and furthermore that $B = \Psi(U)$ is an open ball centered at the origin.
For the rest of this chapter, the letter $B$ without subindex refers
to this particular ball.
\\
Denote by $\Phi : B\to U$ the inverse of $\Psi$
and define $f = V\circ \Phi$, which is a map
from $B$ to $\R$.
Then $u(U) = \graph f|_B$. For $z\in B$ we define
$$
G(z) = (z, f(z)),
$$
so that $u = G\circ\Psi$.
Denote the 
Riemannian metric on $B$ induced by $G$ by
$$
\t g = (\D G)^T (\D G) = I + \D f\otimes\D f,
$$
and the normal to $G$ by
$\t n_u = \frac{\d_1 G\times \d_2 G}{|\d_1 G\times \d_2 G|}$. We have
\begin{equation}
\label{gk-2}
\t n_u = \frac{(-\D f, 1)^T}{\left( 1 + |\D f|^2 \right)^{1/2}},
\end{equation}
because $\det\t g = 1 + |\D f|^2$.

\begin{lemma}
We have $f\in W^{1, \infty}(B)\cap W^{2,2}(B)$
and $f$ satisfies the prescribed Gauss curvature equation
\begin{equation}
\label{gk-1}
\det\D^2 f = K(\Phi) \cdot (1 + |\D f|^2)^2
\end{equation}
almost everywhere on $B$.
\end{lemma}
\begin{proof}
Clearly, $f\in W^{1, \infty}(B)$. Moreover, $f\in W^{2,2}(B)$ by the chain rule
and since $\Phi$ is Bilipschitz and in $W^{2,2}$.
To prove \eqref{gk-1}, note that
$\d_{\a}\d_{\b} G = (0, 0, \d_{\a}\d_{\b} f)^T$.
Therefore, $\t h = \t n_u \cdot \D^2 G$ satisfies
\begin{equation}
\label{gk-3}
\t h = \frac{\D^2 f}{(1 + |\D f|^2)^{1/2}}.
\end{equation}
Taking determinants in \eqref{gk-3}, we see that
\begin{equation}
\label{gk-1b}
\det\D^2 f = (1 + |\D f|^2) \det\t h.
\end{equation}
Using the chain rule, it is easy to verify that 
\begin{equation}
\label{gk-4}
h = (\D\Psi)^T\ \t h(\Psi)\ (\D\Psi) \mbox{ almost everywhere.}
\end{equation}
A similar relation applies to $g$ and $\t g$. Therefore, using Lemma \ref{egregium},
we see that \eqref{gk-1b} implies \eqref{gk-1}.
\end{proof}

Observe that the right-hand side of \eqref{gk-1} is positive and bounded away from zero and infinity. 
Hence the following lemma implies that $f$ is a (locally) convex function.
\begin{lemma}\label{le-sverak}
Let $c$, $R > 0$ and let $\t f\in W^{2,2}(B_R)$ satisfy $\det\D^2\t f \geq c$ 
almost everywhere on $B_R$.
Then $\t f$ is either locally convex on $B_R$
or it is locally concave on $B_R$.
\end{lemma}
\begin{proof}
This is proven in \cite{sverak}. As observed in \cite{sverak}, the results in \cite{IwaniecSverak} 
(which were conjectured in \cite{sverak})
indeed allow to relax the $W^{2,\infty}$ hypothesis in \cite{sverak} to the
$W^{2,2}$ hypothesis used here.
\end{proof}
\begin{lemma}
We have $f\in C^1(B)$.
\end{lemma}
\begin{proof}
We know from Proposition \ref{pro1}
that $n_u : U\to\S^2$ is continuous. And so is $\Phi$.
Since $\t n_u = n_u(\Phi)$, we see that $\t n_u$ is continuous.
Upon scalar multiplication of \eqref{gk-2} with $e_3$, we have
$$
(1 + |\D f|^2)^{-1/2} = n_u(\Phi)\cdot e_3.
$$
Since $n_u(\Phi)\cdot e_3$ is continuous and strictly positive, we conclude that
$(1 + |\D f|^2)^{1/2}$ is continuous. Hence so is $\D f$, by \eqref{gk-2}.
\end{proof}

The following lemma is \cite[Theorem 1']{NikolaevShefel}; cf. also \cite{Caff2}.

\begin{lemma}\label{le-nish}
Let $R > 0$ and $0 < m < M < \infty$ and let
$\t F : B_R\to [m, M]$. If $\t f\in C^0(\o B_R)$ is a convex 
Aleksandrov solution of
$$
\det\D^2\t f = \t F,
$$
then there exists $p\geq 1$ such that $\t f\in W_{loc}^{2,p}(B_R)$. Moreover,
$p\to\infty$ as $M/m\to 1$.
\end{lemma}

\begin{proof}[Proof of Theorem \ref{thm1}]
Denote the right hand side of \eqref{gk-1} by $F$.
Since $F$ is continuous, after possibly shrinking $B$ (and $U$),
the oscillation of $F$ is as small as we please on $B$,
and Lemma \ref{le-nish} implies that there
exists $p > 2$ such that $f\in W_{loc}^{2,p}(B)$.
\\
But $f\in W^{2,p}_{loc}(B)$ implies that  $G\in W_{loc}^{2,p}(B)$.
And \eqref{gk-3} implies that
$\t h\in L_{loc}^p(B)$. Since $\D\Psi\in L^{\infty}(U)$,
from \eqref{gk-4} we deduce that $h\in L^p_{loc}(U)$.
\\
Since the Christoffel symbols are 
smooth and $\D u$ is bounded, we deduce from \eqref{gk-5} that $u\in W_{loc}^{2,p}(U)$.
In particular, by Morrey-Sobolev embedding, there exists a constant $\delta\in (0, 1)$
such that $\Psi\in C^{1,\delta}(U)$ and $f\in C^{1,\delta}(B)$. Since
the Gauss curvature $K$ is Lipschitz on $U$, we have $F\in C^{0,\delta}(B)$.
\\
E.g. by the results in \cite{Schulz-MZ1982}, 
we therefore deduce from $\det\D^2 f = F$ that $f\in C^{2,\delta}(B)$.
Hence \eqref{gk-3} shows that $\t h\in C^{0, \delta}(B)$.
Hence $h\in C^{0, \delta}(U)$ by \eqref{gk-4}.
Thus \eqref{gk-5} implies $u\in C^{2,\delta}(U)$ for some
$\delta\in (0, 1)$.
\\
Hence, for every constant unit vector $e\in\R^3$, the function $v = e\cdot u$
is a $C^{2,\delta}$ solution of the Darboux equation
$$
\det\left( \D^2 v - \Gamma\cdot\D v \right) = K |g| \left( 1 - g^{-1} : (\D v\otimes\D v) \right)
$$
on $U$; here we write $(\Gamma\cdot\D v)_{ij} = \Gamma_{ij}^k \partial_k v$.
This equation is elliptic with respect to $v$. Since $g^{-1}$, $\Gamma$ and $K$ are smooth,
and since $v\in C^{2, \delta}(U)$, we conclude by standard theory 
that $v\in C^{\infty}(U)$ (cf. e.g. \cite{GT01}).
\end{proof}

In closing, note that if merely $g \in C^{2,\alpha}$ for some $\a\in (0, 1)$,
then our arguments show that there is $\delta > 0$ such that $u \in C^{2,\delta}$
for some $\delta\in (0, 1)$.
And if $g \in C^{k,\alpha}$ for some $k \geq 3$, then 
$u \in C^{k+2,\alpha}$, by standard elliptic regularity.

\subsection{A consequence of Theorem \ref{thm1}}

A question arising in problems in thin film elasticity such as the one
addressed in the second part of this paper regards
the existence of solutions $w : U\to\R^3$ of the following 
degenerate PDE system on $U$:
\begin{equation}
\label{liniso}
\d_{\a} u\cdot\d_{\b} w + \d_{\b} u\cdot\d_{\a} w 
= q_{\a\b}\mbox{ for }\a, \b = 1, 2.
\end{equation}
Here $u : U\to\R^3$ is a given $W^{2,2}$ immersion and
$q : U\to\R^{2\times 2}_{\sym}$ is given.
\\
If $u$ is intrinsically convex, a key step in solving \eqref{liniso}
is Theorem \ref{thm1}, as it ensures the ellipticity of the underlying equation.
The other key step is
\cite[Theorem 1.1]{alessandrini}
about unique continuation for elliptic PDE with irregular coefficients.
Combining these two, we obtain the following result:
\begin{proposition}\label{metrch}
Let $\alpha\in (0, 1)$,
let $U\subset \R^2$ be a simply connected domain with $C^2$ boundary
and let $g\in C^{2,\alpha}(\o U, \R^{2\times 2}_{\sym})$ be a Riemannian metric
whose Gauss curvature is positive on $\o U$. Assume that
$u\in W^{2,2}_g(U)\cap W^{2,\infty}(U, \R^3)$.
Then, for every $q \in  W^{1,2} (U, \R^{2\times 2}_{\sym})$,
the system \eqref{liniso}
admits a solution $w\in W^{1,2}(U, \R^3)$.
\end{proposition}
\begin{proof}
A proof for the case $u \in W^{3,\infty}(U,\R^3)$
can be found in \cite[Lemma 5.6]{Lewicka1}; an earlier
proof of a similar (but dual) statement can be found in \cite{LodsMardare}.
Both proofs combine arguments by Weyl presented in \cite{nirenberg} with
a unique continuation result. So do we in the following proof.
\\
As before, $h_{\a\b} = n_u\cdot\d_{\a}\d_{\b} u$
denotes the second fundamental form of $u$. By
$(b_{\a\b})$ we denote the inverse matrix to $(h_{\a\b})$.
In fact, by Theorem \ref{thm1} the matrix $(h_{\a\b})$ is positive definite
everywhere or negative definite everywhere. We assume the former.
\\
We now argue as in \cite[Section 6]{H-stationary} and
introduce the linear operator
$$
T : W^{1,2}(U)\to \left( W^{1,2}_0(U) \right)'
$$
by setting
$$
(T\psi)(\p) = \int_{U} (b_{\a\b}\d_{\a}\psi \d_{\b}\p - 2 H \psi \p )\sqrt{|g|}.
$$
Above, the prime denotes the topological dual and $H$ is the mean curvature of $u$.
For our purposes it is enough to know that $H\in L^{\infty}(U)$ is bounded
from below by a positive constant.
We claim that $T$ is surjective.
\\
In order to prove this, it suffices to show that the dual operator
to $T$
is injective. Let $\p\in W^{1,2}_0(U)$ be such that
$$
\int_{U} (b_{\a\b}\d_{\a}\psi\d_{\b} \p - 2 H \psi \p )\sqrt{|g|} = 0 \mbox{ for all }\psi\in W^{1,2}(U).
$$
We extend $g$ and  $b$ to a simply connected domain $\t U$ containing $\o U$,
in such a way that  $b\in L^{\infty}(\t U)$ is positive definite on $\t U$.
Since $\p\in W^{1,2}_0(U)$, its extension by zero (still denoted  by $\p$) belongs
to $W^{1,2}(\t U)$. Since the restriction of $\psi\in W^{1,2}(\t U)$
to $U$ belongs to $W^{1,2}(U)$, we have
$$
\int_{\t U} (b_{\a\b}\d_{\a}\psi\d_{\b} \p - 2 H \psi \p )\sqrt{|g|} = 0
\mbox{ for all }\psi\in W^{1,2}(\t U).
$$
Since $\p = 0$ on $\t U\setminus\o U$, this implies that
$\p = 0$ on $\t U$, due to \cite[Theorem 1.1]{alessandrini}.
This proves that the dual operator is injective,
hence that $T$ is surjective.
\\
As shown in \cite{nirenberg}, 
the existence of a solution $w\in W^{1,2}(U, \R^3)$ of
\eqref{liniso} is equivalent
to the existence of a solution $\rho\in W^{1,2}(U)$ of $T\rho = f$,
for a suitable $f\in \left( W^{1,2}_0(U) \right)'$ which can
be computed from $u$ and $q$.
\end{proof}

\section{Homogenization for shells} 

We begin by introducing some further notation. Set $Y=[0,1)^2$ and
$\calY = \R^2/\Z^2$. For all $k\in\N\cup\{0\}$ the
set of all $f\in C^k(\R^2)$ with $D^{\a}f(\cdot + z) = D^{\a}f$
for all $z\in\Z^2$ and all multiindices $\a$ of order up to $k$
is denoted by $C^k(\calY)$.
\\
$C^k$ functions with compact support are denoted by a subindex $0$.
For any open set $A$,
we denote by $L^2(\calY)$, $W^{1,2}(\calY)$ and $W^{1,2}(A{\times}\calY)$ the Banach spaces obtained as
closures of $C^\infty(\calY)$
and $C^\infty(\bar A,C^\infty(\calY))$ with respect to the norm in $L^2(Y)$,
$W^{1,2}(Y)$ and $W^{1,2}(A{\times} Y)$,
respectively. An additional dot (e.g. in
$\dot{L}^2(\calY)$) denotes functions with average zero over $\calY$.

\subsection{Surfaces and shells in $\R^3$}

Let $\kappa\in (0, 1)$ and let $\omega\subset\R^2$ be a bounded domain with 
$C^{3, \kappa}$ boundary.
Set $I = (-\tfrac{1}{2},\tfrac{1}{2})$ and $\Omega^h = \omega\times(hI)$, and 
$\Omega = \omega\times I$.
From now on, $S\subset\R^3$ denotes (the relative interior of)
an embedded compact connected oriented surface with boundary.
For convenience we assume that $S$ is parametrized by a single chart.
More precisely, we assume that there exists an open set $V\subset\R^3$ containing
the closure of $S$ and an open set $U\subset\R^3$ containing
$\o\omega\times\{0\}$
and a $C^{3,\kappa}$ diffeomorphism $\Phi : V\to U$ such that
$$
\Phi(S) = \omega\times\{0\}.
$$
Then $\xi : \omega\to S$, defined by $\xi(z) = \Phi^{-1}(z, 0)$,
is a global $C^{3,\kappa}$ chart for $S$.
\\
By $W^{2,2}_{\iso}(S)$ 
we denote the $W^{2,2}$ isometries of the surface $S$ into $\R^3$. 
The space $W^{2,\infty}_{\iso}(S)$ is defined similarly.
Clearly
$
u \in W^{2,2}_{\iso}(S)
$
is equivalent to $u \circ \xi\in W^{2,2}_g (\omega)$,
for $g = (\D\xi)^T(\D\xi)$ the 
Riemannian metric on $\omega$ induced by $\xi$. 
\\
As usual, $TS$ denotes the tangent bundle over $S$ and $NS$ the normal bundle.
A basis of the tangent space $T_x S$ is given by
$$
\tau_{\a}(x) = (\d_{\a}\xi)(\Phi(x))\mbox{ for all }x\in S,
$$
where $\a = 1, 2$.
We view $T_xS$ as a subspace of $\R^3$ and write
$\sigma\cdot\tau$ to denote the scalar product on both spaces.
\\
The dual basis of the tangent space $T_x S$
is denoted by $(\tau^1(x), \tau^2(x))$. So by definition
$$
\tau^{\a}\cdot\tau_{\b} = \delta_{\a\b} \mbox{ on }S,
$$
where $\delta_{\a\b}$ is the Kronecker symbol.
We frequently identify $T_x^* S$ with $T_x S$ via the scalar product.
\\
Define the normal $n : S\to\mathbb{S}^2$ by
$$
n = \frac{\tau_1\times\tau_2}{|\tau_1\times\tau_2|}.
$$
The orthogonal projection onto $T_xS$ is
$$
T_S(x) = I - n(x)\otimes n(x).
$$
The tensor products $TS\otimes TS$ etc. are defined fiberwise.
$T_x^*S\otimes T_x^*S$ will be regarded as a subspace of $\R^{3\times 3}$.
If $E$ and $F$ are vector spaces (or bundles) then
the space of all symmetric products
$$
a\odot b := \frac{1}{2}\left( a\otimes b + b\otimes a\right),
$$
with $a\in E$ and $b\in F$
is denoted by $E\odot F$. 
\\
Sections $B$ of $T^*S\otimes T^*S$ will frequently
be regarded as maps from $S$ into
$\R^{3\times 3}$ via the embedding $\iota$ defined by
	$
	\iota(B) = B(T_S, T_S).
	$
	By definition, $B(T_S, T_S) : S\to\R^{3\times 3}$ takes the vector fields
	$v, w : S\to\R^3$ into the function $x\mapsto B(x)(T_S(x)v(x), T_S(x)w(x))$.
	\\
    For any vector bundle $E$ over $S$ we denote by $L^2(S, E)$ the space
	of all $L^2$-sections of $E$. The spaces $W^{1,2}(S, E)$ etc. are defined similarly.
	For any vector bundle $E$ over $S$ with fibers $E_x$,
	we denote by $L^2(\Y, E)$ the vector bundle over $S$ with
	fibers $L^2(\Y, E_x)$. The bundles $W^{1,2}(\Y, E)$ etc. are defined similarly.
	For example, $L^2$-sections of the bundle $W^{1,2}(\Y, TS)$ are given by
\begin{eqnarray*}
	& & L^2(S, W^{1,2}(\Y, TS)) = \\ & &\hspace{+5ex}\{Z\in L^2(S, W^{1,2}(\Y, \R^3)) : Z(x)\in W^{1,2}(\Y, T_xS)\mbox{ for 
a.e. }x\in S\}.
\end{eqnarray*}
	For a function $f : S\to\R$ its differential $df$ is given by
	$
	df(x)\tau = \D_{\tau} f(x)
	$
	for all $\tau\in T_xS$.
	Here $\D_{\tau} f$ denotes the directional derivative of $f$ in direction 
	of the tangent vector $\tau$. We extend these definitions componentwise to maps into $\R^3$. By $\D$ we denote the usual gradient on $\R^3$ (or on $\R^2$).
\\
As usual, the Weingarten map $\AA$ of $S$ is the differential of the normal, i.e.,
$$
	\AA(x)\tau = (\D_{\tau} n)(x) \mbox{ for all }x\in S,\ \tau\in T_x M.
$$
We extend $\AA(x)$ trivally to $\R^3$
by setting $\AA(x) = \AA(x)\ T_S(x)$.
\\
For an immersion $u : S\to\R^3$ denote by $\AA_u$ the Weingarten map for
the surface $u(S)$.
We define its pullback to $S$ by setting
$$
(u^* \AA_u)\tau = u^*\left( \AA_u D_{\tau}u \right)
$$
for all smooth
tangent vector fields $\tau$ to $S$. Here
by definition, $u^* (D_{\sigma} u) = \sigma$ for all smooth tangent vector fields $\sigma$
to $S$. As in \cite{FJMMshells03} we will encounter the relative Weingarten map
$$
\AA^r_u = u^*\AA_u - \AA.
$$
The nearest point retraction $\pi$ of a tubular neighbourhood of $S$ onto $S$ satisfies
$
\pi (\hx + t\na(\hx)) = \hx 
$
for small $|t|$ and all $x\in S$.
After rescaling the ambient space, we may assume that the curvature of $S$ is as small as we please.
Therefore, we may assume without loss of generality that $\pi$ is
well-defined on a domain containing the closure of the set
$\{\hx+t\na (\hx): \ \hx \in S, -1/2 < t < 1/2 \}$,
and that
$
|Id + t\AA(\hx)|\in (1/2, 3/2)
$
for all $t\in (-\tfrac{1}{2}, \tfrac{1}{2})$ and all $\hx\in S$.
\\
For a subset $\t S \subset S$ and $h\in (0,1]$ we define
$
\t S^h=\{\hx + t\na (\hx): \ \hx \in\t S,\ -h/2 < t < h/2 \}.
$
In particular, the whole shell is, by definition,
$$
S^h = \left\{\hx+t\na (\hx): \ \hx \in S\mbox{ and } t\in (-\frac{h}{2}, \frac{h}{2})\right \}.
$$
We introduce the map $r = \Phi\circ\pi$.
Moreover, we introduce the function $t : S^1\to\R$ by setting
$t(x) = (x - \pi(x))\cdot n(x)$ for all $x\in S^1.$
Then we have the following identity on $S^1$, cf. \cite{Hornungvel12}:
\begin{equation} \label{id1}
d\pi = T_S(\pi) \left( I + t\AA(\pi)T_S(\pi)\right)^{-1}.
\end{equation}
(Here and elsewhere we write $T_S(\pi)$ instead of $T_S\circ\pi$ etc.)
Hence there exists a constant $C$ depending only on $S$
such that
\begin{equation}
\label{Dpi}
\left|
d\pi - (I - t\AA(\pi))T_S(\pi)
 \right|\leq Ct^2\mbox{ on }S^1.
\end{equation}
Abusing notation, maps $f : S \to \R^k$ will often
be extended to $S^1$ by setting
$f = f\circ \pi$. We extend $r$, $T_S$ and $\AA$ in this way, too.
\\	
For functions $f\in L^2(S, W^{2,2}(\Y))$ the expression $\Hessy f$ is the section of
the bundle $L^2\left( \Y, TS\odot TS \right)$ over $S$ given by
$$
(\Hessy f)(x, y) = 
\d_{y_{\a}}\d_{y_{\b}} f(x, y)\ \tau^{\a}(x)\odot\tau^{\b}(x),
$$
where $(\D_y^2 f)_{\a\b} = \d_{y_{\a}}\d_{y_{\b}}f$.
For $v\in L^2(S, W^{1,2}(\Y; \R^2))$ we define
the section $\Defy v$ of the bundle $L^2(\Y, T^*S\odot T^*S)$ by
$$
(\Defy v)(x, y) = (\sym\D_y v(x, y))_{\a\b}\tau^{\a}(x)\odot\tau^{\b}(x).
$$
Here and elsewhere $\D_y$ is the 
gradient in $\Y$ with respect to the variable $y$
(and not some directional derivative).

We define the map $\Xi : \omega\times\R\to\R^3$ by
$$
\Xi(z', z_3) = \xi(z') + z_3n(\xi(z')) \mbox{ for all $z'\in\omega$ and }z_3\in\R.
$$
We define the diffeomorphism $\Theta^h : S^h \to S^1$  by
\begin{equation*} 
\Theta^h = \pi + \frac{t}{h}\ n.
\end{equation*}
Using (\ref{id1}) we can see that
\begin{equation}
\label{dphi}
\nabla\Theta^h = 
\left(T_S + \frac{1}{h}(n\otimes n + t\AA)\right)(I + t\AA)^{-1}
\mbox{ on }S^h.
\end{equation}
For a deformation $u : S^h\to\R^3$ 
its rescaled version $y : S^1\to\R^3$ is defined by
$ 
\ya(\Theta^h) = \ua \mbox{ on }S^h.
$  
We also define the rescaled gradient $\D_h y$ of $y$ by the condition
\begin{equation}
\label{defDh}
(\nabla_h y)\circ \Theta^h = \nabla u \mbox{ on }S^h.
\end{equation}

\subsection{Two-scale convergence on shells}\label{Twoscale}

Recall that $r = \Phi\circ\pi$.
A sequence $(f^h)\subset L^2(S^{1})$ is said to converge weakly two-scale
on $S^1$ to the function $f\in L^2(S^{1},L^2(\calY))$ as $h\to 0$, provided
that the sequence $(f^h)$ is bounded in $L^2(S^{1})$ and
\begin{equation}\label{2skalen}
	\lim_{h\to 0}\int_{S^{1}} f^h(x)\,\rho(x, r(x)/\e)\ud
	x=\int_{S^{1}}\int_{\Y}f(x,y)\,\rho(x,y)\ud y\ud x,
\end{equation}
for all $\rho\in C_0^0(S^{1},C^0(\calY))$.
We write $f^h\wtto f$ to denote weak two-scale convergence.
\\ 
Defining $\t f^h = f^h\circ\Xi$ and $\t f(z, y) = f(\Xi(z), y)$,
and taking
$$
\t\rho(z, y) = \rho(\Xi(z), y)|\det\D\Xi(z)|
$$
a change of variables shows that \eqref{2skalen} is equivalent to
\begin{equation*}
\int_{\Omega} \t f^h(z)\t\rho(z, z'/\e)\ dz \to
\int_{\Omega} \int_{\Y} \t f (z, y)\t\rho(z,y)\ dy\ dz,
\end{equation*}
where $z'$ is the projection of $z$ onto $\R^2$. Hence $f^h\wtto f$ on $S^1$ if and only if
$\t f^h\wtto \t f$ on $\Omega$ in the usual sense.
When $f^h : S\to\R$, then $f^h\wtto f$ on $S$ means, by definition, that the trivial extensions converge weakly two-scale on $S^1$.
In particular, $f^h\wtto f$
on $S$ if and only if $\t f^h\wtto \t f$ on $\omega$.
All these definitions are extended componentwise
to vector-valued maps. For sections $q$, $q^h$ 
of $T^*S\odot T^* S$, we say $q^h\wtto q$
if $q^h(\tau, \sigma)\wtto q(\tau, \sigma)$ for all $\tau, \sigma\in C^1(S, TS)$.
A similar definition applies to other bundles.
\\

{\bf Remarks and Definition.}
\begin{enumerate}[(i)]
\item If $(f^h)\subset L^2(S^{1})$ is bounded, then it has a
subsequence which converges weakly two-scale to some $f \in
L^2(S^{1},L^2(\calY))$. These and other facts can be deduced from the corresponding
results on planar domains, cf. \cite{Allaire-92, Visintin-06}.
\item As usual, an important step in the proof of Theorem 
\ref{tm:glavni} will be to characterize the possible two-scale limits
of scaled gradient fields. For this purpose, 
for $\gamma\in [0, \infty]$
we introduce the following subspaces $\mathcal H_{\gamma}$ of $L^2(S \times I \times \calY, \R^{3 \times 3})$:
\begin{itemize}
\item The space $\mathcal H_0$ is the set of all matrix fields of the form
$\left(\nabla_y\wa_1, \wa_2\right)$, where $\wa_1\in L^2(S,\dot{W}^{1,2}(\mathcal Y,\R^3))$
and $\wa_2\in L^2(S \times I \times \Y, \R^3))$.
\item For $\gamma\in (0, \infty)$, the space $\mathcal H_{\gamma}$ 
is the space of all matrix fields of the form 
$\left(\,\nabla_y\wa_1,\tfrac{1}{\gamma}\partial_3\wa_1\,\right)$, where
$\wa_1\in L^2(S,\dot{W}^{1,2}(I\times\mathcal Y, \R^3))$.
\item The space $\mathcal H_{\infty}$ is the set of all matrix fields of the form 
$(\nabla_y\wa_1,\wa_2)$, where
$\wa_1\in L^2(S \times I, \dot{W}^{1,2}(\mathcal Y, \R^3))$
and $\wa_2\in L^2(S \times I,\R^3)$.
\end{itemize}
\item As shown in \cite[Lemma 4.3]{Hornungvel12}, if $(w^h)\subset W^{1,2}(S^1, \R^3)$
is such that $(w^h)$ and $(\D_h w^h)$ are bounded in $L^2$,
then one can extract a subsequence such that there exists a field
$\bs H_\gamma \in\mathcal H_{\gamma}$ and a map $\wa_0 \in W^{1,2}(S,\R^3)$
with
\begin{equation*}
\nabla_h\wa^h\wtto d\wa_0\ T_S + 
\sum_{i, j = 1}^3(H_\gamma(\pi, t, y))_{ij}\tau^i\otimes\tau^j
\end{equation*}
weakly two-scale on $S^1$,
where $\tau^3 = n$. 
More precisely, $\wa_0$ is the weak limit in $W^{1,2}(S)$ of
$\int_I \wa^h(\cdot + tn)dt$.
\end{enumerate}

\subsection{Energy functionals}

From now on $\e : (0, 1)\to (0, 1)$ denotes a function such
that the limit
$$
\gamma = \lim_{h \to 0} \frac{h}{\e(h)}
$$
exists in $[0,\infty]$.
If $\gamma = 0$, then we will assume, in addition, that $\e^2(h) \ll h$. 
We will often suppress the explicit $h$-dependence in the notation and
simply write $\e$ instead of $\e(h)$.
\\
Let us now fix an energy density function
$$
W : S^{1}\times\R^2\times \R^{3 \times 3}\to [0,\infty]
$$
with the usual properties: 
$W$ is normalized such that
$W(x, y, I) = 0$; moreover, $W$
is continuous in the first argument, $\Y$-periodic in the second
and frame indifferent in the third.
Regarding its growth, we assume that there exist constants $0 < \eta_1\leq\eta_2$ and $\rho > 0$ such that
for all $(x, y)\in S^1\times\Y$ we have
\begin{align*}
	W(x, y, F) &\geq \eta_1\dist^2(\boldsymbol F,\SO 3),\quad\text{ for all }F\in\R^{3 \times 3},
	\\
	W(x, y, F) &\leq \eta_2\dist^2(\boldsymbol F,\SO 3),\quad\text{ for all }F\in\R^{3\times 3}\mbox{ with }
	\dist^2(\boldsymbol F,\SO 3)\leq\rho.
\end{align*}
Finally, we assume that for each $(x, y)\in S^1\times\Y$ there exists a
quadratic form $\mathcal Q(x, y, \cdot) : \R^{3\times 3}\to\R$
such that
\begin{equation}
	\label{esup}
	\esssup_{(x,y) \in S^1 \times \calY}
	\frac{|W(x,y,I+G)-\mathcal{Q}(x,y,G)|}{|G|^2}\to 0 \mbox{ as }G\to 0.
\end{equation}

The following properties of $\mathcal{Q}(\cdot,\cdot,\cdot)$ follow at once from those of $W$ (cf. \cite[Lemma~2.7]{Neukamm-11}): 
the map $\mathcal{Q}(\cdot,y,\cdot)$ is continuous for almost every 
$y\in\R^2$ and the map $\mathcal{Q}(x,\cdot, G)$ is $Y$-periodic
	for all $x\in\Omega$ and all $G\in\R^{3\times 3}$. Moreover,
	for all $x\in\Omega$ and almost every $y\in\R^2$, the map $\mathcal{Q}(x,y,\cdot)$ is quadratic, and for all $\boldsymbol G\in\R^{3\times 3}$ we have
\begin{equation*}
			\eta_1|\sym \boldsymbol G|^2\leq \mathcal{Q}(x,y,\boldsymbol G)=\mathcal{Q}(x,y,\sym \boldsymbol G)\leq 
\eta_2|\sym \boldsymbol G|^2.
		\end{equation*}

The elastic energy per unit thickness
of a deformation $u^h \in W^{1,2}(S^h,\R^3)$ of the shell $S^h$ is given by
$$
J^h(u^h)=\frac{1}{h}\int_{S^h} W\left( \Theta^h(x), r(x)/\e, \nabla u^h(x) \right)\ dx.
$$
In order to express the elastic energy in terms of the $y$ variables, we associate with $y : S^1\to\R^3$
the energy
\begin{eqnarray*}
	I^h(y) &=& \int_{S^1} W\left( x, r(x)/\e, \nabla_h y(x) \right)
	\det\left(I+t(x) \AA(x) \right)^{-1}\ dx \\
	&=& \int_S \int_I W\left( x+t n(x), r(x)/\e, \nabla_h y(x + tn(x)) \right)\ dt\ d\vol_S(x).
\end{eqnarray*}
By a change of variables we have
$$
J^h(u^h) = \frac{1}{h} \int_{S^1} W\left( x, r(x)/\e, \nabla_h y^h(x) \right)\ \left|\det\nabla (\Theta^h)^{-1}(x)
\right|\ dx,
$$
where again $y^h (\Theta^h)=u^h$. 
Using \eqref{dphi} we see that there exists a constant $C$ such
that
\begin{equation*}  
|J^h(u^h) - I^h(y^h)|\leq Ch I^h(y^h).
\end{equation*}

\subsection{Asymptotic energy functionals}

Next we will introduce the asymptotic energy functionals.
In order to do so, we need the definition of the relaxation fields and the cell formulae.
Recall that $a\odot b = \frac{1}{2}(a\otimes b + b\otimes a)$.
We make the following definitions:
\\
Set $D(\mathcal U_0)  =
\dot{W}^{1,2}(\mathcal Y,\R^2)\times \dot W^{2,2}(\Y)\times L^2(I\times\mathcal Y,\R^3)$
and for $(\zeta, \varphi, \mu)\in L^2(S, D(\mathcal U_0))$
define
\begin{align*}
	\mathcal{U}_0 (\zetaa,\varphi, \mu) &= \Defy\zeta + 2\mu_{\a} \tau^{\a}\odot n + \mu_3 n\odot n - t\Hessy\varphi.
\end{align*}
Set $D\left( \mathcal U_{\infty} \right) =
L^2(I, \dot{W}^{1,2}(\mathcal Y,\R^2))\times L^2(I, \dot W^{1,2}(\Y))\times L^2(I, \R^3)$
and for $(\zeta, \rho, c)\in L^2\left(S, D\left( \mathcal U_{\infty} \right)\right)$
define
$$
\mathcal{U}_\infty(\zeta,\rho, c) = \Defy\zeta +
2(\d_{y_{\a}}\rho + c_{\a})\tau^{\a} \odot n + c_3 n\odot n.
$$
For $\gamma\in (0, \infty)$ set $D(\mathcal U_{\gamma}) =
\dot{W}^{1,2}(I\times \mathcal Y;\R^2)\times \dot{W}^{1,2}(I\times \mathcal Y)$
and for $(\zeta, \rho)\in L^2\left(S, D(\mathcal U_{\gamma})\right)$
define
$$
\mathcal{U}_\gamma (\zeta, \rho) = \Defy\zeta +
(\d_{y_{\a}}\rho + \frac{1}{\gamma}\d_3\zeta_{\a}) \tau^{\a}\odot n
+ (\frac{1}{\gamma}\d_3\rho) n\odot n.
$$
By embedding $D(\U_0)$ trivially into $L^2(S, D(\U_0))$,
we can regard $\U_0$ as a map from $D(\U_0)$ into 
$L^2(S, L^2(I\times\mathcal Y, \R^{3\times 3}_{\sym}))$.
\\
For each $x\in S$ the fiberwise action $\U_0^{(x)}$ of $\U_0$ is
$$
\U_0^{(x)}(\zeta, \varphi, \mu) = (\Defy\zeta)(x) 
+ 2\mu_{\a} \tau^{\a}(x)\odot n(x) + \mu_3 n(x)\odot n(x)
- t(\Hessy\varphi)(x),
$$
for all $(\zeta, \varphi, \mu)\in D(\mathcal U_0)$. 
\\
For each $x\in S$ we define $L_0^{(x)}(I\times\Y) = \U_0^{(x)}(D(\U_0))$, i.e.,
$$
L_0^{(x)}(I\times\Y) = \left\{ \U_0^{(x)}(\zeta, \varphi,\mu) : 
(\zeta, \varphi, \mu)\in D(\mathcal U_0)\right\}.
$$
This is a subspace of $L^2(I\times\Y, \R^{3\times 3}_{\sym})$.
We denote by
$L_0(I\times\Y)$ the vector bundle over $S$ with fibers $L_0^{(x)}(I\times\Y)$;
in what follows we will frequently omit the index $(x)$ for the fibers.
The bundles $L_{\gamma}(I\times \Y)$ , for $\gamma \in (0,\infty]$ are defined analogously.
The elements of these spaces are the relaxation fields.
\\
For $\gamma \in [0,\infty]$ and $x \in S$, we define
$
 \mathcal{Q}_{\gamma} (x, \cdot) : T_x^*S \otimes T_x^* S \to\R
$
by setting
\begin{equation*}
\mathcal{Q}_{\gamma}(x, q) = \inf
\int_I\int_{\calY} \mathcal{Q}\Big(x + tn(x), y, p + tq + U(t, y) \Big)\ dy \ dt.
\end{equation*}
Here the infimum is taken over all
$U \in L^{(x)}_{\gamma} (I \times \calY )$ and all
$p \in T_x^*S \otimes T_x^* S$.
\\
Notice that 
$\mathcal{Q}_{\gamma}(x, q)=\mathcal{Q}_{\gamma}(x, \sym q)$
for all $x \in S$ and all $q\in T_x^*S \otimes T_x^* S$.
For $x\in S$ and $q\in T_x^*S\odot T_x^* S$ define  the homogeneous
relaxation (cf. \cite{Lewicka1}):
\begin{equation*}
	\tilde{\mathcal{Q}}(x, t, q) =
	\min_{M\in \R^{3\times 3}_{\sym}} \{ \mathcal{Q}(x+tn(x), M): M(T_S, T_S) = q(T_S, T_S)\}.
\end{equation*}

Then it is easy to see that
\begin{align*}
	\mathcal{Q}_0(x,q) &= \inf \int_{I \times \mathcal{Y}}
	\tilde{\mathcal{Q}}\Big(x+tn(x),y, p + tq
	+ (\Defy\zeta)(x) - t(\Hessy\varphi)(x) \Big)\ dt\ dy,	
\end{align*}
where the infimum is taken over all $\zeta\in \dot W^{1,2}(\Y,\R^2)$, all $\varphi\in\dot W^{2,2}(\Y)$ and all $p \in T_x^*S 
\odot T_x^* S$.
In the case when the material is homogeneous
in the thickness direction, we have
$$
\mathcal{Q}_0(x,q) =
\frac{1}{12} \inf\left\{
\int_{\mathcal{Y}} \tilde{\mathcal{Q}} (x,y,q+ (\Hessy\varphi)(x) ) \ud y :
\varphi \in \dot{W}^{2,2} (\mathcal Y)\right\}.
$$
As in \cite{NeuVel-12}, for all $x \in S$ and all 
$q \in T_x^*S\odot T_x^*S$ we have
\begin{equation*} \lim_{\gamma \to \infty} \mathcal{Q}_{\gamma}(\hx,q)=\mathcal{Q}_{\infty}(\hx,q)
	\mbox{ and }  \lim_{\gamma \to 0} \mathcal{Q}_{\gamma}(\hx,q) =\mathcal{Q}_0 (\hx,q).
\end{equation*}
It is not difficult to show that for all $\gamma \in [0,\infty]$ and $x \in S$ 
the map $\mathcal{Q}_{\gamma}(x,\cdot)$  is quadratic and that there exist $c_1,c_2>0$ such that for all $x \in S$ we have
\begin{equation*}
c_1|\sym q|^2 \leq \mathcal{Q}_{\gamma}(x, q) \leq c_2 |\sym q|^2, 
\quad \forall q \in T_x^*S \otimes T_x^* S.
\end{equation*}

For $\gamma \in [0,\infty]$ we define
$I_\gamma : W^{1,2}(S;\R^3) \to \R$ by setting
$$
I_\gamma (u) =
\begin{cases}
\int_S \mathcal{Q}_\gamma\left(x,\AA^r_u(x) \right) d\vol_S(x) & \quad \textrm{if } u \in W^{2,2}_{\iso}(S), \\
+\infty & \quad \textrm{otherwise}.
\end{cases}
$$

\subsection{Main result}

For a given sequence $(\ua^h) \subset W^{1,2}(S^h;\R^3)$
we continue to define the sequence $(y^h) \subset W^{1,2}(S^1, \R^3)$
of rescaled deformations by $\ya^h(\Theta^h) = \ua^h$.
\\
We recall the compactness result for sequences with finite bending energy,
cf. \cite[Theorem 1]{FJMMshells03} for a proof.
\begin{proposition}
	\label{T:FJM-compactness}
	Let $(u^h)\subset W^{1,2}(S^h,\R^3)$ satisfy 
	\begin{equation} \label{finitebending}
		\limsup_{h \to 0} h^{-2}J^h(\ua^h) < \infty
\end{equation}
Then there exists $u\in W^{2,2}_{\iso}(S)$ such that
(after passing to subsequences
and extending $u$ and $n$ trivially to $S^1$), as $h\to 0$ we have
\begin{align*}
	y^h - \frac{1}{|S^1|}\int_{S^1} y^h\,dx&\to u
	\mbox{ strongly in }W^{1,2}(S^1,\R^3),\\
	\nabla_h y^h&\to Q\mbox{ strongly in }L^2(S^1,\R^{3\times 3}).
\end{align*}
Here $Q \in W^{1,2}(S,\SO 3)$ is determined by the condition that
$Q\tau = \nabla_{\tau} u$ for all smooth tangent vector fields $\tau$ along $S$.
\end{proposition}
We denote by $\t W^{2,2}_{\iso}(S)$ the set of those maps
$u\in W^{2,2}_{\iso}(S)$ 
for which there exist $u_n\in W_{\iso}^{2,\infty}(S)$ 
converging strongly to $u$ in $W^{2,2}$.
The reason to introduce this space is that we are able to construct
the recovery sequence only for limiting deformations $u$ belonging to this
space. Theorem \ref{thm1} plays an essential role in this construction.
\\
The following $\Gamma$-convergence
result is the main result of this chapter: 
\begin{theorem} \label{tm:glavni}
Let $\gamma = \lim_{h \to 0} \frac{h}{\e}$.
If $\gamma=0$ then assume, in addition, that $\e^2 \ll h$. Then the following
are true:
	\begin{enumerate}[(i)]
	\item  \label{tm:glavni-i}
	 Let $(u^h)\subset W^{1,2}(S^h, \R^3)$ be such that
	 \eqref{finitebending} 
and such that $\ya^h - \frac{1}{|S^1|}\int_{S^1} y^h \to u$
strongly in $L^2(S^1)$
for some $u\in L^2(S^1, \R^3)$.
Then 
$$
\liminf_{h \to 0} h^{-2} J^h (\ua^h) \geq I_{\gamma}(u).
$$
\item 
%		
%		Assume, in addition, that $\omega$ is star-shaped with respect to some 
%ball $B \subset \omega$ and that $S$ is convex, i.e.,  there exists $C>0$ such that for every $x\in S$ we have
%		$$  \frac{1}{C} |\tau|^2  \leq \AA(x) \tau\cdot \tau \leq C |\tau|^2,\quad \forall \tau \in T_xS. $$
%
	If, in addition, $S$ is simply connected, then for every $u \in \t W^{2,2}_{\iso}(S)$
		there exist $(\ua^h) \subset W^{1,2}(S^h; \R^3)$
		satisfying (\ref{finitebending}), and such that $y^h \to u$, strongly in $W^{1, 2}(S^1)$.
		Moreover,
		$$
		\lim_{h \to 0} h^{-2} J^h (\ua^h) = I_{\gamma}(u). 
		$$
	\end{enumerate}
\end{theorem}

\subsection{Proof of lower bound}

We consider a sequence $(u^h)\subset  W^{1,2} (S^h,\R^3)$ satisfying 
\begin{equation}
\label{uuuvjet} \limsup_{h \to 0} h^{-2} J^h(u^h) < \infty
\end{equation}
and we set $y^h(\Theta^h) = u^h$.
The following lemma is essentially contained in \cite{FJMMshells03}.
It is a consequence of \cite[Theorem 3.1]{FJM-02} and of the arguments in
\cite{FJM-06}.

\begin{lemma}
	\label{t6f2}
	Define
	\begin{equation*}
	\delta =  \left\{ 
	\begin{aligned} 
	\e, &&  \text{if }  \gamma\in (0, \infty),
	\\
	\lceil\frac{h}{\e}\rceil \e, && \text{if } \gamma = \infty,
    \\
    h, &&  \text{if }\gamma = 0.
	\end{aligned} \right. 
	\end{equation*} 
	Then there exist constants
	$C, c > 0$ such that the following is true: if $h \leq c$ and
	if $u\in W^{1,2}(S^h, \R^3)$, then there exists a map $\t R : \omega\to SO(3)$
	which is constant on each cube $x + \delta Y$ with $x\in\delta\Z^2$ and 
	there exists
	$\t R_s\in W^{1,2}(\omega, \R^{3 \times 3} )$ such that for each $a \in \R^2$ with 
	$|a_1|\leq \delta$ and $|a_2|\leq \delta$ 
	and for each $\t \omega \subset \omega$ with
	$\dist(\t \omega,\partial \omega) > c\delta$ we have:
	\begin{eqnarray*}
&& \|(\nabla_h  y)(\Xi) - \t R\|^2_{L^2(\t \omega \times I)} 
+ \|\t R - \t R_s\|^2_{L^2(\t \omega)}
+h^2 \|\t R - \t R_s\|^2_{L^\infty(\t \omega)}
\\
		&& + h^2\|(\d_1 \t R_s, \d_2 \t R_s )\|^2_{L^2(\t \omega)}
		+\|\t R(\cdot + a) - \t R\|^2_{L^2(\t \omega)}
		\\
		&& \leq C 
		\int_{\Omega} \dist^2 
\left(\nabla_h  y (\Xi), SO(3)\right).
	\end{eqnarray*}
\end{lemma}

\begin{proposition}\label{dvoskalnilimesi}
Let $\gamma \in (0,\infty)$, let $(u^h)$ satisfy (\ref{uuuvjet}) and 
let $u \in W^{2,2}_{\iso}(S)$ be as in the conclusion of 
Proposition \ref{T:FJM-compactness}.
Let $\t \omega\subset\R^2$ be a domain with $C^{1,1}$ boundary
whose closure is contained in $\omega$ and set 
$\t S= \xi(\t\omega)$.
\\
Denote by $\t R^h : \omega \to \SO 3$ the piecewise constant map obtained by applying Lemma \ref{t6f2} to $u^h$ and define $R^h : S^1\to SO(3)$ by
$
R^h = \t R^h\circ r.
$
Define $G^h\in L^2 (S^1,\R^{3 \times 3})$ by
\begin{equation} \label{apstrain}
G^h =\frac{(R^h)^T \nabla_h y^h - I}{h},
\end{equation}
where $y^h(\Theta^h) = u^h$.
Then there exist
$B \in L^2 (\t S, T^*\t S \odot T^*\t S)$ and 
$(\zeta, \rho)\in L^2(\t S, D(\mathcal U_{\gamma}))$ such that
(after passing to subsequences)
\begin{equation}\label{dvoskal-1}
	\sym G^h \wtto B+t\AA^r_{u} + \U_{\gamma}(\zeta,\rho).
\end{equation}
A similar result is true if $\gamma = \infty$ or 
if $\varepsilon^2 \ll h \ll \varepsilon$.
In the former case, $\U_{\gamma}(\zeta,\rho)$ in \eqref{dvoskal-1}
must be replaced by
$\U_{\infty}(\zeta,\rho,c)$, where 
$(\zeta,\rho,c) \in L^2(\t S; \DD(\U_{\infty}))$. In the latter case,
it must be replaced by $\U_{0}(\zeta,\varphi,\mu)$, where
$(\zeta, \varphi,\mu)\in L^2(\t S, D(\mathcal U_0))$.
\end{proposition}

\begin{proof}
Define $\o u^h : S\to\R^3$ by setting
$$
\o u^h (x)= \frac{1}{h}\int_{hI} u^h(x + tn(x)) dt\mbox{ for all }x\in S. 
$$
Let $\t R^h_s : \t\omega\to \R^{3\times 3}$ be the maps obtained
by applying Lemma \ref{t6f2} to $u^h$ and set $R^h_s = \t R^h_s\circ r$.
On   $\t S^h$
 define $z^h$ via
\begin{equation*} 
u^h = \bar{u}^h(\pi) + t (R^h_s n)(\pi) + h z^h.
\end{equation*}
Clearly
$$
\D_n u^h = (R^h_sn)(\pi) + h\D_n z^h.
$$
Let $\tau$ be a
smooth tangent vector field along $S$. Then we have
\begin{align*}
\D_{\tau} u^h &= \D_{\D_{\tau}\pi}\o u^h(\pi) + t(\D_{\D_{\tau}\pi}R^h_s)(\pi) n(\pi)
+ t (R^h_s \AA)(\pi)\D_{\tau}\pi
+ h\D_{\tau}z^h.
\end{align*}
Observe that \eqref{Dpi} implies that $\D_{\tau}\pi$ equals $\tau - t\AA\tau$
up to a term of higher order.
Using this and rewriting the problem in coordinates,
one can now argue as in \cite[Proposition 3.2]{Horneuvel12} to
deduce the claim for $\gamma > 0$. For $\gamma = 0$ one argues as in
\cite[Proposition 3.2]{Vel13}. 
The fields $\mathcal U_{\gamma}$ arise, essentially,
due to the last remark in Section
\ref{Twoscale}. We refer to \cite{Hornungvel12} for details.
\end{proof} 

The remaining proof of the lower bound
follows standard arguments: truncation, Taylor 
expansion and lower semicontinuity of integral functional with respect to two-scale convergence. Thus one obtains a lower bound on every $C^{1,1}$ bounded 
compactly contained
subdomain $\t \omega$ of $\omega$. Exhausting $\omega$ with
a sequence of such subdomains, Theorem \ref{tm:glavni} \eqref{tm:glavni-i} follows.
Details for this argument can be found in \cite{Hornungvel12,Vel13}.

\subsection{Proof of upper bound}

We begin by introducing the `geometric' part of the recovery sequence.

\begin{lemma}\label{geometric portion}
Let $u\in W^{2,\infty}_{iso}(S)$ and define $\nu : S\to\S^2$ by
$$
\nu =
\frac{\D_{\tau_1}u\times \D_{\tau_2}u}{|\D_{\tau_1}u\times \D_{\tau_2}u|}.
$$
Let
$w\in W^{2, \infty}(S, \R^3)$ and define $\mu\in W^{1, \infty}(S, \R^3)$
by
$$
\mu = (\nu\cdot \D_{\tau_1} w )\ \D_{\tau^1}u
+
(\nu\cdot \D_{\tau_2} w )\ \D_{\tau^2}u
$$
and define the deformations $v^h : S^h\to\R^3$ by
$$
v^h = u + t\nu + h\left( w + t\mu \right).
$$
Define $R\in W^{1,\infty}(S, SO(3))$ by
$
R =  \D u\ T_S + \nu\otimes n.
$
Then there exist $Y^h\in L^{\infty}(S^h, \R^{3\times 3})$
with $\|Y^h\|_{L^{\infty}(S^h)}\leq Ch^2$ such that
$$
dv^h\odot R = I + h\ du\odot dw + t \AA_u^r + Y^h.
$$
\end{lemma}
\begin{proof}
First of all observe that $R$ indeed takes values in $SO(3)$,
because $u$ is an isometric immersion. Now
set $P(x) = d\pi(x)$ and
let $\tau$, $\sigma$ be smooth tangent vector fields to $S$. We have
$$
\D_{\tau} v^h = \D_{P\tau} u(\pi) +
t \D_{P\tau}\nu(\pi) + h\left( \D_{P\tau}w(\pi) + t\D_{P\tau}\mu(\pi) \right)
$$
Since by definition
$$
\D_{\sigma}u\cdot \D_{P\tau}\nu = \sigma\cdot u^* \AA_u P\tau
$$
and since $\D_{P\tau}u\cdot \D_{\sigma} u = P\tau\cdot\sigma$ because
$u$ is an isometric immersion, we compute
\begin{align*}
\D_{\tau}v^h\cdot \D_{\sigma} u(\pi)
&= P\tau\cdot\sigma + t\sigma\cdot (u^*\AA_u)P\tau
\\
&+ h \D_{P\tau}w(\pi)\cdot \D_{\sigma}u(\pi) +
ht \D_{P\tau}\mu(\pi)\cdot (\D_{\sigma} u)(\pi). 
\end{align*}
Now observe that $P$ equals
$
(I - t \AA)T_S
$
plus an error which on $S^h$ is uniformly controlled by $h^2$. 
Hence there exist $\t Y^h$ with $\|\t Y^h\|_{L^{\infty}(S^h)}\leq Ch^2$
such that
\begin{align*}
\D_{\tau}v^h\cdot \D_{\sigma} u(\pi)
= \tau\cdot\sigma &+ t\sigma\cdot \AA^r_u \tau
\\
&+ h \D_{\tau}w(\pi)\cdot \D_{\sigma}u(\pi) +
\t Y^h(\sigma, \tau).
\end{align*}
After symmetrizing
we obtain the claim for tangential vector fields.
\\
On the other hand,
$
\D_n v^h = \nu + h\mu
$
and $Rn = \nu$. So $(d v^h\odot R)(n,n) = 1$. And for $\tau$ as above
and using $\nu\cdot \D_{\tau} u = 0$, we conclude
\begin{align*}
2(dv^h\odot R)(n, \tau) &=
\D_n v^h\cdot R\tau + 
\D_{\tau} v^h\cdot Rn
\\
&= h\mu\cdot \D_{\tau}u + h\D_{\tau} w\cdot \nu + ht\D_{\tau}\mu\cdot \nu.
\end{align*}
The first two terms on the right cancel due to the definition
of $\mu$, and the last term satisfies
the required bound.
\end{proof}

Before proceeding to prove Theorem \ref{tm:glavni} (ii),
we include the following remarks, which motivate our choice of recovery
sequences.

\paragraph{Remarks.}
\begin{enumerate}[(i)]
\item The actual recovery sequence differs in two respects
from the one used in \cite{FJMMshells03} for homogeneous materials:
\\
Firstly, it has to take into account the
inhomogeneities in the material. It will be of the form
$$
\t v^h = v^h + \textrm{relaxation part},
$$
with $v^h$ as in the lemma.
\\
Secondly, the spatial dependence of the energy density
makes it necessary to choose a nonzero displacement $w$ in Lemma
\ref{geometric portion} which generates a prescribed
first order change of the metric. This is the field $B$
arising in \eqref{dvoskal-1}.
In order to recover this field $B$,
we will have to choose $w$ in Lemma \ref{geometric portion} to be
a solution of the PDE system $B = du \odot dw.$
The existence of such a displacement
$w$ is ensured by Proposition \ref{metrch}.
\item Theorem \ref{tm:glavni} applies to multilayered materials
(cf. \cite{Schmidt-07} for the corresponding problem for plates)
as a very particular case.
In that situation, the relaxation part is trivial as in the homogeneous
case. However, the
second effect mentioned above still plays a role. Therefore,
Proposition \ref{metrch}
is essential in that simpler situation as well,
and so is its key ingredient Theorem \ref{thm1}.
\end{enumerate}

\begin{proof}[Proof of Theorem \ref{tm:glavni} (ii)] 
As in \cite{Hornungvel12}, by approximation it is enough to prove
the claim for $u\in W^{2,\infty}_{iso}(S)$
and, thanks also to Proposition \ref{metrch},
for all $B$ of the form $B = du \odot dw$ with $w\in  W^{2,\infty}(S, \R^3)$.
\\
We will use the same notation as in the statement of Lemma \ref{geometric portion};
in particular the definition of $v^h$ in terms of $w$ and $u$.
Moreover, we set $\sigma^{\a} = \D_{\tau^{\a}}u$.
\\

{\bf Case  $\gamma \in (0,\infty)$.}
Let $\zeta\in C_0^1(S,\dot{C}^1(I \times \mathcal{Y},\R^2))$
and $\rho\in  C_0^1(S,\dot{C}^1(I \times \mathcal{Y}))$
and define rescaled deformations $y^h : S^1\to\R^3$ by the following
equation on $S^h$:
\begin{eqnarray*}
y^h(\Theta^h) &=& v^h 
+ h\varepsilon \zeta_{\alpha}\left(\pi,\frac{t}{h},
\frac{r}{\e}\right)\sigma^{\alpha}
+ h\varepsilon\rho\left(\pi,\frac{t}{h},\frac{r}{\e}\right)\nu. 
\end{eqnarray*}
Lemma \ref{geometric portion} implies that on $S^1$
\begin{equation}
\label{zadnje1}
\sym \left(R^T\nabla_h y^h \right) = I + hB + th \AA^r_{u} 
+h\mathcal{U}_{\gamma}(\zeta,\rho)\left(x,t, \frac{r}{\e}\right)+o(h),
\end{equation}
where $\lim_{h \to 0} \| \tfrac{o(h)}{h}\|_{L^{\infty}}=0$.
\\
By frame invariance of $W$ and using \eqref{esup},
we deduce from \eqref{zadnje1} that
$$
\frac{1}{h^2} W\left(\cdot, \frac{r}{\e}, 
\nabla_h y \right)
\to
\mathcal{Q}\left(\cdot, \frac{r}{\e}, \AA^r_{u}
+ B+\mathcal{U}_{\gamma}(\zeta,\rho)\left(\cdot,t,\frac{r}{\e}\right)\right),
$$
pointwise on $S^1$.
From this we readily deduce 
$$
\lim_{h \to 0} h^{-2} I^h(y^h)=
\int_{S}\int_{I \times \mathcal{Y}} 
\mathcal{Q}\Big(\cdot + t n, y,\AA^r_{u} + B +
\mathcal{U}_{\gamma}(\zeta,\rho)(\cdot, t, y)\Big)\ dy\ dt\ d\vol_S.
$$

{\bf Case $\gamma=\infty$.} 
This is similar to the previous case. So we only state the formula for
the recovery sequence.
For $\zeta\in C^1_0(S,C^1_0(I,\dot{C}^1(\mathcal Y,\R^2)))$
and $\rho \in C^1_0(S, C^1_0(I, \dot C^1(\Y)))$ and $c\in C^1_0(S,C^1_0(I, \R^3))$,
we define $y^h : S^1\to\R^3$ by the following
equation on $S^h$:
\begin{eqnarray*}
	y^h(\Theta^h) 
	&=& v^h + h\varepsilon\zeta_{\alpha}\left(\pi, \frac{t}{h}, \frac{r}{\e}\right) \sigma^{\alpha}
+ h\e
\rho\left(\pi,\frac{t}{h},\frac{r}{\e}\right)\ \nu
\\
& &
+ 2h^2\left(\int_0^{ t/h} 
 c_\alpha(x,s)\ ds\right)\ \sigma^\alpha
+ h^2\left(\int_0^{ t/h} c_3(x,s)\ ds\right)\ \nu. 
\end{eqnarray*}
\\
 
{\bf Case $\e^2 \ll h \ll \e$.}
For 
$\zeta\in C^1_0 (S,\dot{C}^1(\mathcal Y,\R^2) )$
and
$\varphi\in C^2_0(S,\dot C^2(\Y))$
and
$\mu\in C^1_0(S,C^1_0(I\times \mathcal Y,\R^3))$
we define $y^h : S^1\to\R^3$ by the following
equation on $S^h$:
\begin{eqnarray*}
	y^h(\Theta^h)&=& v^h +
h\varepsilon \zeta_{\alpha}\left(\pi, \frac{r}{\e}\right)\ \sigma^{\alpha}
	+\e^2 \varphi\left(\pi,r/\e\right)\nu
	- { t\e} \partial_{y_{\alpha}} \varphi\left(\pi, \frac{r}{\e}\right) 
\sigma^{\alpha}
\\
& &
+ 2h^2 \left(\int_0^{ t/h} \mu_{\alpha}\left(\pi, s, \frac{r}{\e}\right)
\ ds\right)\ \sigma^\alpha
+ h^2\left( \int_0^{ t/h} \mu_3(\pi, s, r/\e)\ ds\right)\ \nu.
\end{eqnarray*}
In this case the expression $R^T\nabla_h y^h$
will contain a term of order $\e$, which is much greater than $h$.
After symmetrizing, however, it vanishes as in \cite{Vel13}.
Adapting the arguments from that
paper, we therefore obtain the desired claim. We leave the details to the 
interested reader.
\end{proof}

\paragraph{Acknowledgements.} PH was supported by the DFG; the warm
hospitality at the University of Zagreb is gratefully acknowledged.
IV was supported by Croatian Science Foundation grant no. 9477.

\vspace{1cm}
\def\cprime{$'$}

\bibliographystyle{acm}

\end{document}